\documentclass[12pt]{article}
\usepackage{amssymb,amsmath,amsthm, amsfonts}
\usepackage{graphicx}
\usepackage{epsfig}
\usepackage{tikz}

\def\disp{\displaystyle}

\def\dref#1{(\ref{#1})}

\theoremstyle{plain}
\newtheorem{theorem}{Theorem}[section]
\newtheorem{lemma}{Lemma}[section]

\newtheorem{corollary}{Corollary}[section]

\theoremstyle{definition}

\newtheorem{remark}{Remark}[section]

\numberwithin{equation}{section}

\begin{document}

\title{\bf Boundedness and large time behavior  in a  higher-dimensional
Keller--Segel system with singular sensitivity and
logistic source}

\author{
Jiashan Zheng$^{a}$\thanks{Corresponding author.   E-mail address:
 zhengjiashan2008@163.com (J.Zheng)}, Gui Bao$^{a}$
 \\
    School of Mathematics and Statistics Science,\\
     Ludong University, Yantai 264025,  P.R.China \\
}
\date{}

\maketitle \vspace{0.3cm}
\noindent
\begin{abstract}
This paper focuses on the following  Keller-Segel system with  singular sensitivity and
logistic source
$$
 \left\{\begin{array}{ll}
  u_t=\Delta u-\chi\nabla\cdot(\frac{u}{v}\nabla v)+ au-\mu u^2,\quad
x\in \Omega, t>0,\\
 \disp{ v_t=\Delta v- v+u},\quad
x\in \Omega, t>0\\
 \end{array}\right.\eqno(\star)
$$
in a smoothly bounded domain
 $\Omega\subset\mathbb{R}^N(N\geq1)$, with zero-flux boundary conditions,
 where $a>0,\mu>0$ and $\chi>0$ are given constants. If $\chi$ is    small enough,  then, for all reasonable regular initial data,
a corresponding initial-boundary value problem for
$(\star)$  possesses a global
classical solution  $(u, v)$  which is
 {\bf bounded} in $\Omega\times(0,+\infty)$. Moreover, if $\mu$ is large enough, the solution $(u, v)$ exponentially
converges to the constant stationary solution $(\frac{a}{\mu },
\frac{a}{\mu })$ in the norm of $L^\infty(\Omega)$ as
$t\rightarrow\infty$.
\end{abstract}

\vspace{0.3cm}


\newpage
\section{Introduction}
Chemotaxis is the directed movement of cells as a response to gradients of the concentration of a chemical signal
substance (see Winkler et al.  \cite{Bellomo1216} and Keller and Segel \cite{Keller2710,Keller79}).
%
%
 Following experimental works of Adler (see  Adler et al. \cite{Adleromo1216,Adlessromo1216}), in 1971,
Keller and Segel (\cite{Keller79}) introduced a phenomenological model to
 capture this kind of behaviour,
the
general form of which reads
\begin{equation}
 \left\{\begin{array}{ll}
  u_t=\Delta u-\nabla\cdot(u\chi_0(v)\nabla v),\quad
x\in \Omega, t>0,\\
 \disp{\tau v_t=\Delta v +u- v},\quad
x\in \Omega, t>0,\\
 \end{array}\right.\label{1.ssddddessderrddffttrrfff1}
\end{equation}
where $\tau\in\{0,1\},$ the function $\chi_0$ measures the chemotactic sensitivity,
%
$u$ denotes the cell density and $v$ describes the concentration of the chemical signal which is
directly produced by cells themselves.  
In the last 40 years, a variety of chemotaxis models have been extensively studied with
various mechanisms from the cells diffusivity, the chemotactic sensitivity, and the cells growth-death (see \cite{Bellomo1216,Hillen79}).
We refer to the review papers \cite{Bellomo1216,Hillen79,Horstmann2710}
for detailed descriptions of the models and their developments. The main issue of the investigation in the mathematical
analysis of system \dref{1.ssddddessderrddffttrrfff1} is whether the solutions of the models are bounded or blow-up (see  e.g., \cite{Horstmann2710,Nanjundiahffger79312,Winkler793} for  $\chi_0:=\chi>0$ and \cite{Bilerss1216,Fujiessss3344,Lankeit2233791,Winkler44455678887,Stinnerrr44621,Nagaiddsssdddssddxxss}
for  $\chi_0:=\frac{\chi}{v}$).
%
For example, if $\chi_0(v) := \chi>0,$ it is known that for all reasonably regular initial
data 
the solutions of the corresponding Neumann initial boundary value problem for \dref{1.ssddddessderrddffttrrfff1} are global and remain
bounded when either $N= 1$, or $N = 2$ and
$\int_{\Omega}u_0 < 4\pi$, or $N\geq 3$ and the initial is sufficiently small (\cite{Osaki79,Winkler792,Horstmann791}).
%
%
However, according to Fujie and  Senba (\cite{Fujiessss3344}),
the sensitivity function $\chi_0(v)$ can not always be a constant, for example,
 in accordance with the Weber-Fechner's
law of stimulus perception in the process of chemotactic response, the sensitivity function $\chi_0(v)$ will be chosen by $\chi_0(v) = \frac{\chi}{v}$.
When $\chi_0(v) = \frac{\chi}{v}$,
it is known that all
{\bf radial classical} solutions of  system \dref{1.ssddddessderrddffttrrfff1} are global-in-time if either $N\geq 3$ with $\chi <\frac{2}{N-2}$, or $N= 2$ with $\chi > 0$ arbitrary
(see Nagai and   Senba \cite{Nagaiffhov1230}). If $N \geq 2$, Fujie \cite{Fujieffl79} proved that
there exist globally {\bf bounded}
classical solutions of system \dref{1.ssddddessderrddffttrrfff1} if  $\chi<\sqrt{\frac{2}{N}}$. The proof of boundedness of solutions for  system \dref{1.ssddddessderrddffttrrfff1} 
relying   on the second equation actually ensures  a positive pointwise {\bf lower bound} for $v$. The  {\bf lower bound} for $v$ can  be obtained  by
the {\bf lower bound} for $\int_{\Omega}u$, which is a clear result for  \dref{1.ssddddessderrddffttrrfff1}. One can see \cite{Fujieffl79}  for more  details.
For more  results with
 various {\bf sensitivity} functions, we refer to \cite{Fujiessddd1,Lankeit2233791,ffggLankeitsdddLankeit796}.

Apart from the aforementioned system,  in order  to describe
the spontaneous growth of cells,
a source of logistic type is included in \dref{1.ssddddessderrddffttrrfff1}  (see Winkler \cite{Winkler37103} and see also \cite{Tello710} and Zheng et al. \cite{Zhengjmaa22}).
%
Motivated by the above works,,
in this paper, we deal with the fully parabolic Keller--Segel
system with singular sensitivity and logistic
source
%
\begin{equation}
 \left\{\begin{array}{ll}
  u_t=\Delta u-\nabla\cdot(u\chi_0(v)\nabla v)+au-\mu u^2,\quad
x\in \Omega, t>0,\\
 \disp{ v_t=\Delta v +u- v},\quad
x\in \Omega, t>0,\\
 \disp{\frac{\partial u}{\partial \nu}=\frac{\partial v}{\partial \nu}=0},\quad
x\in \partial\Omega, t>0,\\
\disp{u(x,0)=u_0(x)},\quad  v(x,0)= v_0(x),~~
x\in \Omega\\
 \end{array}\right.\label{111.ssderrfff1}
\end{equation}
in a bounded domain $\Omega\subset R^N(N\geq1)$ with smooth boundary $\partial\Omega$. Our primary interest is in the case in which  
$a>0$ as well as $\mu > 0$
and
\begin{equation}
 \begin{array}{ll}
  \chi_0(v)=\frac{\chi}{v},~~\mbox{for}~~~
v>0\\
 \end{array}
\label{1.ssdelffllsgggsderrddfftjjjtrddffrfff1}
\end{equation}
with a constant $\chi > 0$.

For \dref{111.ssderrfff1} with {\bf $\chi_0(v):=\chi>0$}, a large quantities  of  literatures  are devoted to investigating boundedness and blow-up of the solutions (see e.g., Cie\'{s}lak et al.
 \cite{Cie72,Cie794,Cie791}, Burger et al. \cite{Burger2710},
  Calvez and  Carrillo \cite{Calvez710},  Keller and Segel \cite{Keller2710,Keller79},
   Horstmann et al. \cite{Horstmann2710,Horstmann79,Horstmann791},
 Osaki \cite{Osaki79},
    Painter and Hillen \cite{Painter79}, Perthame \cite{Perthame710},
Rascle and Ziti \cite{Rascle710}, Wang et al. \cite{Wang76}, Winkler \cite{Winkler555ani710,Winkler79,Winkler37103,Winkler715,Winkler793,Winkler79312vv}, Xiang \cite{Xiangggg} and
Zheng \cite{Zhengsssddsseedssddxxss}). In fact,
for any $\mu> 0,$ it is also shown that the logistic source can prevent blow up
whenever $N\leq 2$, or $\mu$ is sufficiently large (see Osaki and  Yagi \cite{Osaki79}, Osaki et al. \cite{Osakix391},  Winkler \cite{Winkler37103}, Zheng et al. \cite{Zhengjmaa22}).

The mathematical challenge to \dref{111.ssderrfff1} with {\bf $\chi_0(v)=\frac{\chi}{v}$} is that we must avoid the singular value $v = 0$. Therefore, in order to
show the global existence and {\bf boundedness} to problem \dref{111.ssderrfff1}, we should gain a positive pointwise {\bf lower bound} for $v$, which
is a well-known fact for problem \dref{1.ssddddessderrddffttrrfff1}, due to the variation-of-constants formula for $v$ and the fact that
$$
\int_{\Omega}{u (x,t)}=\int_{\Omega}{u_0(x)}>0~\mbox{for all}~~ t>0.
$$
 As for logistic sources contains   in \dref{111.ssderrfff1} with {\bf quadratic absorption}, however,
nothing seems to be known in this direction so far (see  Zhao and  Zheng \cite{Zhaoffgg} and Winkler et al. \cite{Fujiesssssss3344} for $N=2$).
Up to now, however, global existence results seem to be available only
for certain simplified variants such as e.g. the {\bf two-dimensional} analogue of \dref{111.ssderrfff1} (see  Zhao and  Zheng \cite{Zhaoffgg} and Winkler et al. \cite{Fujiesssssss3344}).
Therefore, very few results 
 appear to be available on system \dref{111.ssderrfff1} with such singular sensitivities and logistic source (see, e.g., Zhao and Zheng \cite{Zhaoffgg} and Winkler et al. \cite{Fujiesssssss3344}). In fact, in the spatially {\bf two-dimensional} case, the knowledge about  systems of type \dref{1.ssdelffllsgggsderrddfftjjjtrddffrfff1} is expectedly much
further developed.
The parabolic-elliptic system \dref{111.ssderrfff1} (the second equation of \dref{111.ssderrfff1} is replaced by $\Delta v =v-u$)
 was considered in \cite{Fujiesssssss3344}, where it was obtained that there
exists a unique globally bounded classical solution whenever
\begin{equation}\label{x1sssss.731426kkklll677gg}
a>\left\{
\begin{array}{ll}
\displaystyle{\frac{\chi^2}{4}~~\mbox{if}~~0<\chi\leq2},\\
\displaystyle{\chi-1~~\mbox{if}~~\chi>2.}\\
\end{array}
\right.
\end{equation}
Recently, if $a$ satisfies \dref{x1sssss.731426kkklll677gg},
 Zhao and Zheng (\cite{Zhaoffgg}) obtained the global bounded classical solution
 for the {\bf fully parabolic} system \dref{111.ssderrfff1} 
 in the
{\bf 2-dimensional} setting. When the (nonlinear)  logistic source is included, $\chi$
 satisfies suitable restrictions, Zhao and Zheng (\cite{Zhao23444ffgg}) further  studied the global
existence and boundedness of very weak solutions under the assumption that $\Omega$ is a  bounded {\bf convex} domain. However,
to the best
of our knowledge, it is yet unclear whether for $\Omega$ is a {\bf non-convex} domain and $N\geq3$, the solution of the problem is  bounded or not.
With some carefully analysis, the purpose of the present work is to investigate the convergence of all solution
components in \dref{1.ssdelffllsgggsderrddfftjjjtrddffrfff1} under some conditions, possibly
 involving the initial data or the interaction between chemotactic cross-diffusion and the limitation
 of cell growth. Without any restriction on the {\bf space dimension}, 
the first object of the present paper is to address the global {\bf boundedness} of solutions to \dref{1.ssdelffllsgggsderrddfftjjjtrddffrfff1}. Our main result in this respect is the following.
\begin{theorem}\label{theorem3}
Assume
that 
the initial data $(u_0,v_0)$ fulfills
\begin{equation}\label{x1.731426677gg}
\left\{
\begin{array}{ll}
\displaystyle{u_0\in C^0(\bar{\Omega})~~\mbox{with}~~u_0\geq0~~\mbox{in}~~\Omega~~\mbox{and}~~u_0\not\equiv0,~~x\in\bar{\Omega}},\\
\displaystyle{v_0\in W^{2,\infty}(\Omega)~~\mbox{with}~~v_0>0~~\mbox{in}~~\bar{\Omega},~~~\mbox{and}~~\frac{\partial v_0}{\partial \nu}=0,~~x\in\partial\Omega}.\\
\end{array}
\right.
\end{equation}
 Let $\Omega\subset\mathbb{R}^N(N\geq1)$ be a smooth bounded  domain.
  If $a$ and  $\chi$
satisfies \dref{x1sssss.731426kkklll677gg}
 and
\begin{equation}\label{x1sssss.7314kkkkkk26kkklll677gg}
\left\{
\begin{array}{ll}
\displaystyle{\chi>0~~\mbox{if}~~N=1},\\
\displaystyle{0<\chi<\sqrt{\frac{2}{N}}~~\mbox{if}~~N\geq2},\\
\end{array}
\right.
\end{equation}
respectively,
 then there exists a unique pair $(u, v)$ of non-negative functions:
  $$
\left\{\begin{array}{rl}
&\disp{u\in C^0(\bar{\Omega}\times[0,\infty))\cap C^{2,1}(\bar{\Omega}\times(0,\infty)),}\\
&\disp{v\in C^0(\bar{\Omega}\times[0,\infty))\cap C^{2,1}(\bar{\Omega}\times(0,\infty)),}\\
\end{array}\right.
$$
 which solves \dref{111.ssderrfff1}  classically. Moreover, the solution of \dref{111.ssderrfff1} is bounded
in $\Omega\times (0,\infty)$.
%
%
%
%
%
%
%
%
%
\end{theorem}


\begin{remark}

 %

(i) We should point that
in view of singular sensitivities,
 a variation of Maximal Sobolev Regularity can not be used to solve problem \dref{111.ssderrfff1}
 (see \cite{Cao33444334,Zhengjmaa22}), since, it is hard to estimate  $\int_{\Omega}|\nabla v|^{2q}$ by using the boundedness of $\int_{\Omega}|\Delta v|^{q}$ for $N\geq3$.

 (ii) We should point  that the idea of this paper can be also solved with other types of models, e.g.
 an chemotaxis-growth model with indirect
attractant production
 and singular sensitivity in a Keller-Segel-fluid system with logistic source.

\end{remark}


Going beyond these boundedness statements, to the best of our knowledge, only few results on the  large time behavior were
studied in  chemotaxis model \dref{111.ssderrfff1}
%
(see e.g. Winkler  et al. \cite{Tao44621,Winkler79312vv,Winkler79312}, Galakhov et al. \cite{Galakhov1230}).
For instance, 
if {\bf $\chi_0(v)=\chi>0$},
Winkler (\cite{Winkler79312vv}) found that the   solutions of one dimensional parabolic--elliptic models \dref{111.ssderrfff1}
 may become large at intermediate time scales provided
that $a=\mu<1$.
On the other hand, 
 in \cite{Winkler79312} it was found
that all the solutions of the  Keller-Segel system  \dref{111.ssderrfff1} with $a= 1$ and $\chi_0(v)=\chi>0$ converge to
$(\frac{1}{\mu}, \frac{1}{\mu})$ exponentially for a suitable small value of $\frac{\chi}{\mu}$ and {\bf convex domain} $\Omega$.
 Recently, by applying  a variation of {\bf Maximal Sobolev Regularity},
\cite{Xiangggg}   and \cite{Zhengssddghhhhjmaa22}  (see also \cite{Cao33444334}) improve the results of
 \cite{Winkler79312} to
a bounded non-convex domain. As compared to this, the  {\bf large time behavior}  to Keller-Segel system \dref{111.ssderrfff1} with {\bf singular sensitivity} seems to be much less understood.
To the best of our knowledge, not even one dimensional   result  for large time behavior seems available, due to the challenges lies in this problem.
%


Motivated by the above works, it seems natural  that our second result,
addressing asymptotic homogenization of all solution components, requires $\mu$ to be
appropriately large. Our result in this direction can be stated as follows:

\begin{theorem}\label{theoremssdddffggdd3} Assume the hypothesis of Theorem \ref{theorem3}
holds. 
 Then there exists $\mu_0 > 0$ depending  on $a,\chi ,v_0$ and $\Omega$
  such that if
\begin{equation}\label{fvgbccvvhnjmkfgbdffrhnjjn6291}\mu>\mu_0,
\end{equation}
one can find $\gamma>0$ as well as $t_0$ and $C>0$ such that the global classical solution $(u, v)$ of \dref{111.ssderrfff1} satisfies
\begin{equation}\label{fvgbccvvhnjmkfgbdffrhnn6291}
\|u(\cdot,t)-\frac{a}{\mu }\|_{L^\infty(\Omega)}\leq Ce^{-\gamma t},  ~~\mbox{for all}~~t>t_0
\end{equation}
and
\begin{equation}\label{ffgfvgbccvvhnjmkfghhjjgbhnn6291}
\|v(\cdot,t)-\frac{a}{\mu }\|_{L^{\infty}(\Omega)}\leq Ce^{-\gamma t},  ~~\mbox{for all}~~t>t_0.
\end{equation}
\end{theorem}
%
%
%
%
%
%
%
%
\begin{remark}


%


 (i) Theorem \ref{theoremssdddffggdd3} extends   the results of  Theorem 1.1   (\cite{Winkler79312}),
where   the convexity of $\Omega$
 required in \cite{Winkler79312}.

 (ii) It should be mentioned that
   the idea of this paper can be also solved with other types of models, e.g.
 an chemotaxis-growth model with indirect
attractant production and singular sensitivity.

%
\end{remark}


It is worth to remark the main idea underlying the  proof of our results.
The key step  to the proof of Theorem \ref{theorem3} is to establish a positive uniform-in-time lower bound for $v,$
which is equivalent to obtain
$\inf_{0\leq t<\infty} \|u(¡¤, t)\|_{L^1(\Omega)} > 0$ (see Lemma \ref{lemmssssdddssssssa45ddfffff630223}), and can be transformed to build
the global boundedness for 
a weighted integral of the form $\int_{\Omega}u^{-p}v^{-q}dx $
introduced
for system \dref{111.ssderrfff1} with suitable $p, q > 0$ to be determined (see Lemmas \ref{lemmsssa45ddfffff630223}--\ref{lemmssssdddsa45ddfffff630223}).
The technical advantage of small values of $\chi$ (see \dref{x1sssss.7314kkkkkk26kkklll677gg}) is that these will allow us to pick some $\kappa>\frac{N}{2}$,
$q_0\in (0,\frac{N}{2})$
and $C_0>0$
 such that
$$
\begin{array}{rl}
&\disp{\int_{\Omega}u^{\kappa}v^{-q_0}\leq \frac{C_0}{\mu}~~~\mbox{for all}~~t\in(0,T_{max}),}
\\
\end{array}
$$
so that, implies the boundedness of  $L^{\frac{N}{2}+\varepsilon}(\Omega)$ by using  the variation-of-constants formula.
Then we  use  the standard estimate for Neumann semigroup  and the standard Alikakos--Moser iteration (see e.g.  Lemma A.1 of \cite{Tao794}) to show Theorem \ref{theorem3}.

In order
to prove  Theorem \ref{theoremssdddffggdd3} we will find a nonnegative function $F$ satisfying
$$
\begin{array}{rl}
&\disp{F'(t):=\frac{d}{dt}\left(\int_{\Omega}(U-1-\ln U)+\frac{L}{2} \disp\int_{\Omega}V^2\right)\leq-G_0(\int_{\Omega}(U-1)^2+\frac{L}{2}\int_{\Omega}V^2)}
\\
\end{array}
$$
 with some suitable positive numbers $L$ and $G_0$ (see Lemma \ref{lemmaddffffdfffgg4dddd5630}) depending on the positive pointwise {\bf lower bound} of $v$.
Then,
by means of an analysis of the above
inequality and the uniform H\"{o}lder estimates (see Lemma \ref{fghfbglemma4563025xxhjklojjkkkgyhuissddff}),
one can establish
$\lim_{t\rightarrow+\infty}(\|u(\cdot,t)-\frac{a}{\mu }\|_{L^\infty(\Omega)}+\|v(\cdot,t)-\frac{a}{\mu }\|_{L^\infty(\Omega)})=0$
(see Lemmas \ref{fghfbglemma4563025xxhjklojjkkkgyhuissddff} and \ref{fhhghfbgddddlemma4563025xxhjklojjkkkgyhuissddff}).
By interpolation, we can thus assert the claimed uniform exponential stabilization
property.
%
We can thereupon make use of the interpolation  and the spatial regularity the solution $(u,v)$ (Lemma \ref{fghfbglemma4563025xxhjklojjkkkgyhuissddff})
to show
that the above convergence actually takes place at an exponential rate (the proof of Theorem \ref{theoremssdddffggdd3}).

\section{Preliminaries}

%
%
In this section,
we first state several elementary
lemmas which will be needed later.
 \begin{lemma}(Page 126 of \cite{Nirenbergfgger79312})\label{lemma4sddd1ffgg}
Let $\Omega$
 be a bounded Lipschitz domain in $R^N$, $p, q, r, s \geq 1, j,m \in \mathbb{N}_0$ and $\alpha\in [\frac{j}{m}, 1]$
satisfying $\frac{1}{p} = \frac{j}{m} + (\frac{1}{r}-\frac{m}{N})\alpha+ \frac{1-\alpha}{q}$. Then there are positive constants $C_1$ and $C_2$ such that for all
functions $w\in L^q(\Omega)$ with $\nabla w \in L^r(\Omega),$ $w\in L^s(\Omega)$,
 %
%
$$\| D^jw\|_{L^{p}(\Omega)} \leq C_{1}\|D^mw\|_{L^{r}(\Omega)}^{\alpha}\|w\|^{1-\alpha}_{L^{q(\Omega)}}+C_2\|w\|_{L^{s}(\Omega)}.$$
\end{lemma}

\begin{lemma}(\cite{Fujie3344,Winkler79312,Winkler792,Zhengghhhhjmaa22})\label{llssdrffmmggnnccvvccvvkkkkgghhkkllvvlemma45630}
Let $(e^{\tau\Delta})_{\tau\geq0}$ be the Neumann heat semigroup in $\Omega$, and $\lambda_1 > 0$ is the
first nonzero eigenvalue of $-\Delta$ in  $\Omega\subset \mathbb{R}^\mathbb{N}$ under the Neumann boundary condition. Then
there exist $c_i=c_i(\Omega)(i=1,2,3,4)$ depending on $\Omega$ such that the following estimates hold.
(i) If $1 \leq q \leq p\leq \infty$, then
$$
\| e^{\tau\Delta}\varphi\|_{L^p(\Omega)} \leq c_1(1+\tau^{-\frac{N}{2}(\frac{1}{q}-\frac{1}{p})})e^{-\lambda_1\tau }\|\varphi\|_{L^q(\Omega)} ~~~\mbox{for all}~~~ \tau > 0~~~\mbox{and any}~~~\varphi\in L^{q}(\Omega)~~\mbox{and}~~\int_{\Omega}\varphi=0.
$$

(ii) If $1 \leq q \leq p\leq \infty$, then
$$
\|\nabla e^{\tau\Delta}\varphi\|_{L^p(\Omega)} \leq c_2(1+\tau^{-\frac{1}{2}-\frac{N}{2}(\frac{1}{q}-\frac{1}{p})})e^{-\lambda_1\tau }\|\varphi\|_{L^q(\Omega)} ~~~\mbox{for all}~~~ \tau > 0~~~
$$
holds and any $\varphi\in L^{q}(\Omega)$.

(iii) If $2 \leq q \leq p< \infty$, then
$$
\|\nabla e^{\tau\Delta}\varphi\|_{L^p(\Omega)} \leq c_3(1+\tau^{-\frac{N}{2}(\frac{1}{q}-\frac{1}{p})})e^{-\lambda_1\tau }\|\nabla \varphi\|_{L^q(\Omega)} ~~~\mbox{for all}~~~ \tau > 0
$$
holds and any $\varphi\in W^{1,p}(\Omega)$.

(iv) If $1 < q \leq p\leq \infty$, then
\begin{equation}
\begin{array}{rl}
&\| e^{\tau\Delta}\nabla\cdot\varphi\|_{L^p(\Omega)} \\
\leq& c_4(1+\tau^{-\frac{1}{2}-\frac{N}{2}(\frac{1}{q}-\frac{1}{p})})\|\varphi\|_{L^q(\Omega)} ~\mbox{for all}~ \tau > 0\\
\end{array}
\label{gggghjjccmmllffvvggcvvvvbbjjkkdffzjscz2.5297x9630xxy}
\end{equation}
holds for all $\varphi\in (L^q(\Omega))^N$.
\end{lemma}

The following local existence result is rather standard, since a similar reasoning in
 \cite{Cie72,Wang76,Wang72,Zheng0,Zhaoffgg}. Therefore,
 we omit it here.
\begin{lemma}\label{lemma70}
 Let $\Omega\subset\mathbb{R}^N(N\geq1)$ be a smooth bounded  domain.
Assume that the nonnegative functions $u_0$ and $v_0$ satisfies \dref{x1.731426677gg}.
%
%
 Then for any $a\in R$ and $\mu>0,$
 there exist
  a maximal existence time $T_{max}\in(0,\infty]$ and a pair  of  nonnegative functions 
 $$
\left\{\begin{array}{rl}
&\disp{u\in C^0(\bar{\Omega}\times[0,T_{max}))\cap C^{2,1}(\bar{\Omega}\times(0,T_{max})),}\\
&\disp{v\in C^0(\bar{\Omega}\times[0,T_{max}))\cap C^{2,1}(\bar{\Omega}\times(0,T_{max})),}\\
\end{array}\right.
$$
 which solves \dref{111.ssderrfff1}  classically and satisfies $u,v>0$
  in $\Omega\times(0,T_{max})$.
%
Moreover, if  $T_{max}<+\infty$, then
$$
\|u(\cdot, t)\|_{L^\infty(\Omega)}+\| v(\cdot, t)\|_{W^{1,\infty}(\Omega)}\rightarrow\infty~~ \mbox{as}~~ t\nearrow T_{max}.
$$
\end{lemma}

\begin{lemma}(\cite{LankeitsdddLankeit796})\label{lemma630}
Let $T\in(0,\infty]$, let $y\in C^1((0, T )) \cap C^0([0, T ))$, $B > 0, A > 0$ and the nonnegative function $h \in C^0([0, T ))$ satisfy
%
%
\begin{equation}\label{x1.73142hjkl}
\begin{array}{ll}
\displaystyle{
 y'(t)+Ay(t)\leq h(t)~~~\mbox{and}~~~\int_{(t-1)_+}^{t}h(s)ds\leq B ~~\mbox{for a.e.}~~t\in(0,T)}.\\
\end{array}
\end{equation}
Then
%
%
$$y(t)\leq y_0+\frac{B}{1-e^{-A}}~~\mbox{for all}~~t\in(0,T).$$
\end{lemma}

\section{The boundedness  and classical  solution of \dref{111.ssderrfff1}}

\subsection{Some well-known result about  \dref{111.ssderrfff1}}
In order to discuss the boundedness  and classical  solution of \dref{111.ssderrfff1}, firstly,  we will
recall some well-known result about the  solutions to \dref{111.ssderrfff1}.

\begin{lemma}\label{wsdelemma45}
Under the assumptions in
 Lemma  \ref{lemma70}, we derive that
there exists a positive constant 
$\lambda$ independent of $a$ and $\mu$
such that the solution of \dref{111.ssderrfff1} satisfies
%
%
\begin{equation}
\int_{\Omega}{u}+\int_{\Omega} {v^2 }+\int_{\Omega}|\nabla {v }|^2 \leq \lambda~~\mbox{for all}~~ t\in(0, T_{max})
\label{cz2.5ghju48cfg924ghyuji}
\end{equation}
and
\begin{equation}
\int_{(t-1)_{+}}^{t}\int_{\Omega}[|\nabla {v }|^2+u^2 + |\Delta {v }|^2]\leq \lambda~~\mbox{for all}~~ t\in(0, T_{max}).
\label{223455cz2.5ghjddffgggu48cfg924ghyuji}
\end{equation}
\end{lemma}
\begin{proof}
From integration of the first equation in \dref{111.ssderrfff1} we obtain
\begin{equation}
\begin{array}{rl}
\disp\frac{d}{dt}\disp\int_{\Omega}u =&\disp{\int_\Omega   (au -\mu u ^2) ~~\mbox{for all}~~ t\in(0, T_{max}),}\\
\end{array}
\label{111811cz2.5114114}
\end{equation}
which 
 implies that
\begin{equation}
\begin{array}{rl}
\disp\frac{d}{dt}\disp\int_{\Omega}u \leq&\disp{a\int_\Omega u -\frac{\mu}{|\Omega|} \left(\int_\Omega u \right)^2
 ~~\mbox{for all}~~ t\in(0, T_{max})}\\
\end{array}
\label{1ssdd11811cz2.5114114}
\end{equation}
by using  the Cauchy-Schwarz inequality.
Hence, employing  the Young inequality to  \dref{1ssdd11811cz2.5114114} and  integrating the resulted inequality   in time, we derive  that there exists
 a positive constant $C_1$ 
 such that
\begin{equation}
\int_{\Omega}{u } \leq C_1~~\mbox{for all}~~ t\in(0, T_{max}).
\label{s3344cz2.5ghju48cfg924ghyuji}
\end{equation}
For  each $t\in(0, T_{max})$, integration with respect to time 
results in
\begin{equation}
\begin{array}{rl}
&\disp{\int_{(t-1)_{+}}^{t}\int_{\Omega}u^2 \leq C_2}\\
\end{array}
\label{bnmbncz2.5ghhdderrffjuyuivvbssdddeennihjj}
\end{equation}
by \dref{s3344cz2.5ghju48cfg924ghyuji}.
Now,
multiplying the second  equation of \dref{111.ssderrfff1}
  by $-\Delta v $, integrating over $\Omega$ and using  the Young inequality, 
 we get
$$
\begin{array}{rl}
\disp\frac{1}{{2}}\frac{d}{dt}\|\nabla v \|^{{2}}_{L^{{2}}(\Omega)}+\int_{\Omega}|\Delta v |^2+\int_{\Omega}|\nabla v |^2
=&\disp{-\int_\Omega  u \Delta v  }\\
\leq&\disp{\frac{1}{2}\int_\Omega  u ^2+\frac{1}{2}\int_\Omega  |\Delta v |^2 ~~\mbox{for all}~~ t\in(0, T_{max}),}\\
\end{array}
$$
from   Lemma \ref{lemma630}  we infer that
\begin{equation}
\int_{\Omega}|\nabla {v }|^2 \leq C_3~~\mbox{for all}~~ t\in(0, T_{max})
\label{cz2.5ghju48cfg92dffff4ghdffffhhyuji}
\end{equation}
and
\begin{equation}
\begin{array}{rl}
&\disp{\int_{(t-1)_{+}}^{t}\int_{\Omega}[|\nabla {v }|^2+ |\Delta {v }|^2]\leq C_4}\\
\end{array}
\label{bnmddfgghghbncz2.5ghhjufffffyuivvbssdddeennihjj}
\end{equation}
by \dref{bnmbncz2.5ghhdderrffjuyuivvbssdddeennihjj}.
Next, testing the second  equation of \dref{111.ssderrfff1}
  by $v $, we conclude that
\begin{equation}
\int_{\Omega}v ^2 \leq C_5~~\mbox{for all}~~ t\in(0, T_{max}).
\label{cz2.5ghju48cfg924ghdffffggggffhhyuji}
\end{equation}
by applying \dref{bnmbncz2.5ghhdderrffjuyuivvbssdddeennihjj}.
Now, collecting \dref{s3344cz2.5ghju48cfg924ghyuji}--\dref{cz2.5ghju48cfg924ghdffffggggffhhyuji} yields to \dref{cz2.5ghju48cfg924ghyuji}
and \dref{223455cz2.5ghjddffgggu48cfg924ghyuji}.
\end{proof}

\subsection{A
lower bounded estimate of $v$}

In order to deal with {\bf the singular sensitivity}, in this subsection,
we will derive a
lower bounded estimate of $v$.
To achieve this, 
we  transform this
into time-independent lower bound for
$\int_{\Omega}u^{-\alpha}$ for some $\alpha>0$. Indeed,
we firstly
conclude  a bound on
$\int_{\Omega}u^{p}v^{q}$ with  some
negative exponents $p$ and $q$.
\begin{lemma}\label{lemmsssa45ddfffff630223} Let $\Omega\subset\mathbb{R}^N(N\geq1)$ be a smooth bounded  domain.
    Let $(u,v)$ be a solution to \dref{111.ssderrfff1} on $(0,T_{max})$.  Then
for all $\tilde{p}, \tilde{q} \in \mathbb{R}$, on $(0, T_{max})$ we have
\begin{equation}
\begin{array}{rl}
&\disp{\frac{d}{dt}\int_{\Omega}u^{\tilde{p}}v^{\tilde{q}}}
\\
=&\disp{-\tilde{p}(\tilde{p}-1)\int_{\Omega}u^{\tilde{p}-2}v^{\tilde{q}}|\nabla u|^2+[\tilde{p}(\tilde{p}-1)\chi-2\tilde{p}\tilde{q}]\int_{\Omega}u^{\tilde{p}-1}v^{\tilde{q}-1}\nabla u\cdot\nabla v}
\\
&+\disp{[-\tilde{q}(\tilde{q}-1)+\tilde{p}\tilde{q}\chi ]\int_{\Omega}u^{\tilde{p}}v^{\tilde{q}-2}|\nabla v|^2+[a\tilde{p}-\tilde{q}]\int_{\Omega}u^{\tilde{p}}v^{\tilde{q}}}
\\
&\disp{-\mu \tilde{p} \int_{\Omega}u^{\tilde{p}+1}v^{\tilde{q}}+ \tilde{q}\int_{\Omega}u^{\tilde{p}+1}v^{\tilde{q}-1}~~ \mbox{for all}~~  t\in(0,T_{max}).}
\\
\end{array}
\label{jjcz2.511ssdddkkkkdfgg41ssfffffsddggggddfgddfggggddffff14}
\end{equation}
\end{lemma}
\begin{proof}
Proceeding analogously to Lemma 2.3 of \cite{Winkler44455678887}, we can prove the desired identity.
\end{proof}

With the help of Lemma \ref{lemmsssa45ddfffff630223}, by using the idea of Lemma 3.2 of \cite{Zhaoffgg},  we can estimate $\int_{\Omega}u^{-p}v^{-q}$ (for some negative exponents $p$ and $q$) in the following  format:
%

\begin{lemma}\label{lemmssssdddsa45ddfffff630223} Let $\Omega\subset\mathbb{R}^N(N\geq1)$ be a smooth bounded  domain
 and $(u,v)$ be a solution to \dref{111.ssderrfff1} on $(0,T_{max})$.  Then for $a \in \mathbb{R}$, there exist $p\in(0,1),C$ and
 $q > q_{1,+}:=\frac{p+1}{2}(\sqrt{1+p\chi^2}-1)$ such that
\begin{equation}
\begin{array}{rl}
&\disp{\frac{d}{dt}\int_{\Omega}u^{-p}v^{-q}\leq (q-ap)\int_{\Omega}u^{-p}v^{-q}+C~~ \mbox{for all}~~  t\in(0,T_{max}).}
\\
\end{array}
\label{caaaz2.511ssdddkkkkdfgsdfsdffssdddfffffffg41sddggggddfgddfggggddffff14}
\end{equation}
\end{lemma}
\begin{proof}
Firstly, choosing $\tilde{p}:=-p>0$ and $\tilde{q}:=-q>0$ and in Lemma \ref{lemmsssa45ddfffff630223}, we obtain
\begin{equation}
\begin{array}{rl}
&\disp{\frac{d}{dt}\int_{\Omega}u^{-p}v^{-q}}
\\
=&\disp{-p(p+1)\int_{\Omega}u^{-p-2}v^{-q}|\nabla u|^2+[p(p+1)\chi-2pq]\int_{\Omega}u^{-p-1}v^{-q-1}\nabla u\cdot\nabla v}
\\
&+\disp{(pq\chi-q(q+1))\int_{\Omega}u^{-p}v^{-q-2}|\nabla v|^2+(q-ap)\int_{\Omega}u^{-p}v^{-q}}
\\
&+\disp{\mu p \int_{\Omega}u^{-p+1}v^{-q}- q\int_{\Omega}u^{-p+1}v^{-q-1}~~ \mbox{for all}~~  t\in(0,T_{max}).}
\\
\end{array}
\label{cz2.511ssdddkkkkdfgg41ssfffffsddggggddfgddfggggddffff14}
\end{equation}
Next, by the  Young inequality, the second term  of \dref{cz2.511ssdddkkkkdfgg41ssfffffsddggggddfgddfggggddffff14} can be estimated by
\begin{equation}
\begin{array}{rl}
&\disp{[p(p+1)\chi-2pq]\int_{\Omega}u^{-p-1}v^{-q-1}\nabla u\cdot\nabla v}
\\
\leq&\disp{p(p+1)\int_{\Omega}u^{-p-2}v^{-q}|\nabla u|^2+\frac{p[(p+1)\chi-2q]^2}{4(p+1)}\int_{\Omega}u^{-p}v^{-q-2}|\nabla v|^2}
\\
\end{array}
\label{cz2.511ssdddkkkkdfgg41sddggggddfgddfggggddffff14}
\end{equation}
for all $t\in(0,T_{max}).$
Inserting \dref{cz2.511ssdddkkkkdfgg41sddggggddfgddfggggddffff14} into \dref{cz2.511ssdddkkkkdfgg41ssfffffsddggggddfgddfggggddffff14} implies that
\begin{equation}
\begin{array}{rl}
&\disp{\frac{d}{dt}\int_{\Omega}u^{-p}v^{-q}}
\\
\leq&\disp{\{\frac{p[(p+1)\chi-2q]^2}{4(p+1)}+pq\chi-q(q+1)\}\int_{\Omega}u^{-p}v^{-q-2}|\nabla v|^2}
\\
&+\disp{[q-ap]\int_{\Omega}u^{-p}v^{-q}+\mu p \int_{\Omega}u^{-p+1}v^{-q}- q\int_{\Omega}u^{-p+1}v^{-q-1}~~ \mbox{for all}~~  t\in(0,T_{max}).}
\\
\end{array}
\label{cz2.511ssdddkkkkdfgg41sddggggddfssdddgddfggggddffff14}
\end{equation}
 Now, denote
$$4(p+1)f(p;q,\chi):=\tilde{f}(q)=-4q^2-4(p+1)q+p(p+1)^2\chi^2,$$
where $$f(p;q,\chi):=\frac{p[(p+1)\chi-2q]^2}{4(p+1)}+pq\chi-q(q+1).$$
Therefore, $$f(p;q,\chi) < 0$$  by the Vi\`{e}te formula and $q > q_{1,+}=\frac{p+1}{2}(\sqrt{1+p\chi^2}-1).$
Combine with \dref{cz2.511ssdddkkkkdfgg41sddggggddfssdddgddfggggddffff14} to get
\begin{equation}
\begin{array}{rl}
&\disp{\frac{d}{dt}\int_{\Omega}u^{-p}v^{-q}\leq (q-ap)\int_{\Omega}u^{-p}v^{-q}+\mu p \int_{\Omega}u^{-p+1}v^{-q}- q\int_{\Omega}u^{-p+1}v^{-q-1}~~ \mbox{for all}~~  t\in(0,T_{max}).}
\\
\end{array}
\label{cz2.511ssdddkkkkdfgsdffffg41sddggggddfgddfggggddffff14}
\end{equation}
Take $p\in(0,1)$, in view of the Young inequality and using \dref{cz2.5ghju48cfg924ghyuji}, we conclude that
there exists some $C_1 > 0$, such that
\begin{equation}
\begin{array}{rl}
\mu p \disp\int_{\Omega}u^{-p+1}v^{-q}\leq &\disp{q\int_{\Omega}u^{-p+1}v^{-q-1}+(\frac{\mu p}{q+1})^{q+1} \int_{\Omega}u^{(1-p)}}
\\
\leq &\disp{q\int_{\Omega}u^{-p+1}v^{-q-1}+C_1,}
\\
\end{array}
\label{cz2.511ssdddkkkddfffkdfgsdffffg41sddggggddfgddfggggddffff14}
\end{equation}
from which \dref{caaaz2.511ssdddkkkkdfgsdfsdffssdddfffffffg41sddggggddfgddfggggddffff14} immediately follows by \dref{cz2.511ssdddkkkkdfgsdffffg41sddggggddfgddfggggddffff14} and some basic  calculation.
\end{proof}
Thanks
to the H\"{o}lder inequality and  $L^p$-$L^q$ for the Neumann heat semigroup, Lemma \ref{lemma70} directly entails a
uniform lower bound for  $v$ in $\Omega$ with $a$ satisfying \dref{111.ssderrfff1}.
\begin{lemma}\label{lemmssssdddssssssa45ddfffff630223} Let $\Omega\subset\mathbb{R}^N(N\geq1)$ be a smooth bounded  domain
 and $(u,v)$ be a solution to \dref{111.ssderrfff1} on $(0,T_{max})$.
 If $a$
satisfies \dref{x1sssss.731426kkklll677gg},
then there exists a positive constant $\eta_0$ independent of $\mu$ such that
\begin{equation}\label{eqxhddfffgfgghhj45xx12112}
v(x,t)\geq \eta_0~~\mbox{for all}~~(x,t)\in\Omega\times(0,T_{max}).
\end{equation}
\end{lemma}
\begin{proof}
Let $\delta_1 = \frac{1}{2}\inf_{x\in\Omega} v_0(x).$
In view of Lemma \ref{lemma70}, 
there exists $t_0\in (0, T_{max})$, such that 
%
%
\begin{equation}\label{eqxhssdddddfffgfgghhj45xx1sddd2112}
v(x, t) > \delta_1~~\mbox{for all}~~(x,t)\in\Omega\times(0,T_{0}].
\end{equation}
So we only need to prove \dref{eqxhddfffgfgghhj45xx12112}
for $t\in (t_0, T_{max})$. In fact, let
%
%
$$\tilde{g}(p):=\frac{p+1}{2}(\sqrt{1+p\chi^2}-1)-ap,~~p>0.$$
Due to $a$ satisfy \dref{x1sssss.731426kkklll677gg},
 we derive that $$(0, 1) \cap (p_{\tilde{g},-}, p_{\tilde{g},+})\neq\varnothing$$
by using the Vi\`{e}te formula again,
where $$p_{\tilde{g},\pm}=\frac{2a^2+2a-\chi^2\pm 2a\sqrt{(1+a)^2-\chi^2}}{\chi^2}.$$
%
%
%
%
%
%
%
Taking $\alpha\in (0, \min\{\frac{pN}{qN-2q+N},p\})$, then
\begin{equation}-\frac{N}{2}(1-\frac{p-\alpha}{q\alpha})>-1
\label{cz2.511ssdddkkkddfffkdfgsdfffsddffffg41sddggggddfgddfggggddffff14}
\end{equation}
and
\begin{equation}
\begin{array}{rl}
&\disp{\int_{\Omega}u^{-\alpha}\leq \left(\int_{\Omega}u^{-p}v^{-q}\right)^{\frac{\alpha}{p}}\left(\int_{\Omega}v^{\frac{q\alpha}{p-\alpha}}\right)^{\frac{p-\alpha}{p}}}
\\
\end{array}
\label{cz2.511ssdddkkddfffgggkkdfgsdffffg41sddggggddfgddfggggddffff14}
\end{equation}
by the H\"{o}der inequality. Integrating \dref{caaaz2.511ssdddkkkkdfgsdfsdffssdddfffffffg41sddggggddfgddfggggddffff14} from $t_0$ to $t$ yields
\begin{equation}
\begin{array}{rl}
&\disp{\int_{\Omega}u^{-p}v^{-q}\leq e^{(q-ap)(t-t_0)}\int_{\Omega}u^{-p}(x,t_0)v^{-q}(x,t_0)+C_1}
\\
\end{array}
\label{caaaz2.511ssdddkkkkdfgsdfsdfffhhhffffffg41sddggggddfgddfggggddffff14}
\end{equation}
with some $C_1 > 0$.

On the other hand, by Lemma \ref{llssdrffmmggnnccvvccvvkkkkgghhkkllvvlemma45630}  with \dref{cz2.5ghju48cfg924ghyuji} and notation $ \bar{u}=\bar{u}(s) :=
\frac{1}{|\Omega|}\int_{\Omega}u$ (for any $s\geq0$), we have
\begin{equation}\begin{array}{rl}
&\| v(\cdot,t)\|_{L^\frac{q\alpha}{p-\alpha}(\Omega)}\\
\leq&\disp \| e^{t(\Delta-1)}v_0\|_{L^\frac{q\alpha}{p-\alpha}(\Omega)} +\int_{0}^{t}\| e^{(t-s)(\Delta-1)}(u(\cdot,s)-\bar{u})\|_{L^\frac{q\alpha}{p-\alpha}(\Omega)} ds+\int_{0}^{t}\| e^{(t-s)(\Delta-1)}\bar{u}\|_{L^\frac{q\alpha}{p-\alpha}(\Omega)} ds\\
\leq&\disp c_1\| v_0\|_{L^\infty(\Omega)} +c_1\int_{0}^{t}(1+(t-s)^{-\frac{N}{2}(1-\frac{p-\alpha}{q\alpha})})e^{-(\lambda_1+1)(t-s) }
\|u(\cdot,s)-\bar{u}\|_{L^1(\Omega)} ds\\
&\disp+ c_1\lambda\int_{0}^{t}e^{-(\lambda_1+1)(t-s) } ds\\
\leq&\disp  C_2~~~\mbox{for all}~~t\in(t_0,T_{max})\\
\end{array}
\label{2ssdddd23455cz2.ddff5ghjsdddddffgggu48cfg9sdddd24ghyuji}
\end{equation}
with some $C_2> 0,$ where $c_1$ is the same as Lemma \ref{llssdrffmmggnnccvvccvvkkkkgghhkkllvvlemma45630}.
Here we have used the fact that
$$\int_{0}^{t}[1+(t-s)^{-\frac{N}{2}(1-\frac{p-\alpha}{q\alpha})}]e^{-(\lambda_1+1)(t-s) } ds\leq \int_{0}^{\infty}[1+(t-s)^{-\frac{N}{2}(1-\frac{p-\alpha}{q\alpha})}]e^{-(\lambda_1+1)(t-s) }  ds<+\infty,$$
due to \dref{cz2.511ssdddkkkddfffkdfgsdfffsddffffg41sddggggddfgddfggggddffff14}.

Combine \dref{cz2.511ssdddkkddfffgggkkdfgsdffffg41sddggggddfgddfggggddffff14}--\dref{2ssdddd23455cz2.ddff5ghjsdddddffgggu48cfg9sdddd24ghyuji} to know there exists  $C_3 > 0$ such that
\begin{equation}
\begin{array}{rl}
&\disp{\int_{\Omega}u^{-\alpha}(x,t)\leq C_3~~~\mbox{for all}~~t\in(t_0,T_{max})}
\\
\end{array}
\label{cz2.511ssdddkkddfffgggdffffkkdfgsdffffg41sddggggddfgddfggggddffff14}
\end{equation}
and hence
\begin{equation}
\begin{array}{rl}
\disp\int_{\Omega}u(x,t)\geq&\disp{ |\Omega|^{\frac{\alpha+1}{\alpha}}\left(\int_{\Omega}u^{-\alpha}\right)^{-\frac{1}{\alpha}}}\\
\geq&\disp{|\Omega|^{\frac{\alpha+1}{\alpha}}C_3^{-\frac{1}{\alpha}}}
\\
:=&\disp{\delta_2~~~\mbox{for all}~~t\in(t_0,T_{max})}
\\
\end{array}
\label{cz2.511ssdddkddffffkddfffgggdffffkkdfgsdffffg41sddggggddfgddfggggddffff14}
\end{equation}
by the H\"{o}lder inequality.
The representation of $v$ as
\begin{equation}\begin{array}{rl}
v(\cdot,t)=&\disp e^{t(\Delta-1)}v_0 +\int_{0}^{t}e^{(t-s)(\Delta-1)}u(\cdot,s) ds~~~\mbox{for all}~~t\in(t_0,T_{max})\\
\end{array}
\label{ffff1.ssddghhhjddfffjffsjkkksderrfff1}
\end{equation}
makes it possible to apply well-known estimates for the Neumann heat-semigroup $\{e^{t\Delta}\}_{t\geq0}$, which provides a positive constant $\delta_3$
such that
\begin{equation}\begin{array}{rl}
v(\cdot,t)=&\disp e^{t(\Delta-1)}v_0 +\int_{0}^{t}e^{(t-s)(\Delta-1)}u(\cdot,s) ds\\
\geq&\disp \int_{0}^{t}\frac{}{(4\pi(t-s))^{\frac{N}{2}}}e^{-[(t-s)+\frac{(\mbox{diam}\Omega)^2}{4(t-s)}]}\int_{\Omega}u(x,s) ds\\
\geq&\disp \delta_2\int_{0}^{t_0}\frac{}{(4\pi \sigma)^{\frac{N}{2}}}e^{-[\sigma+\frac{(\mbox{diam}\Omega)^2}{4\sigma}]}\\
:=&\disp{\delta_3~~~\mbox{for all}~~t\in(t_0,T_{max})~~~\mbox{and}~~x\in \Omega}
\\
\end{array}
\label{1.ssddghhhjddfffjffsjkkksderrfddfffff1}
\end{equation}
by using \dref{ffff1.ssddghhhjddfffjffsjkkksderrfff1} and \dref{cz2.511ssdddkddffffkddfffgggdffffkkdfgsdffffg41sddggggddfgddfggggddffff14}.
Let $\eta_0 = \min\{\delta_1, \delta_3\}$ to complete the proof.
\end{proof}



In order to obtain a bound for $u$ with respect to the norm in $L^\infty(\Omega)$, we need to
%
 obtain an $L^p(\Omega)$-estimate for $u$, for some $p > \frac{N}{2}$.
%
To this end, 
we  transform this
into time-independent upper
bound for
$\int_{\Omega}u^{\beta}$ for some $\beta>0$. In fact,
we firstly
conclude  a bound on
$\int_{\Omega}u^{p}v^{-q}$ with  some
positive  exponents $p$ and $q$.

\begin{lemma}\label{lemkddddddkkkmsssa45630223} Let $\Omega\subset\mathbb{R}^N(N\geq2)$ be a smooth bounded  domain.
Assume that $\chi<1.$ Let $a$ satisfy  \dref{x1sssss.731426kkklll677gg} and $(u,v)$ be a solution to \dref{111.ssderrfff1} on $(0,T_{max})$. Then
for all $p\in(1,\frac{1}{\chi^2})$, for each $q\in (q_{2,-}(p), q_{2,+}(p))$, one can find $C>0$  independent of $\mu$  such that
\begin{equation}
\begin{array}{rl}
&\disp{\int_{\Omega}u^{p}v^{-q}\leq \frac{C}{\mu}~~~\mbox{for all}~~t\in(0,T_{max}),}
\\
\end{array}
\label{cz2.511ssdddkkddfddfffffffgggdffffkkdfgsdffffg41sddggggddfgddfggggddffff14}
\end{equation}
  where
\begin{equation}
q_{2,\pm}(p):=q_{2,\pm}=\frac{p-1}{2}(1\pm\sqrt{1-p\chi^2}).\label{czddfff2.511ssdddkkddfddfffffffgggdffffkkdfgsdffffg41sddggggddfgddfggggddffff14}
\end{equation}
\end{lemma}
\begin{proof}
Firstly, choosing $\tilde{p}:=p>1$ and $\tilde{q}:=-q>0$ and in Lemma \dref{lemmsssa45ddfffff630223}, we obtain
\begin{equation}
\begin{array}{rl}
&\disp{\frac{d}{dt}\int_{\Omega}u^{p}v^{-q}}
\\
=&\disp{-p(p-1)\int_{\Omega}u^{p-2}v^{-q}|\nabla u|^2+[2pq+p(p-1)\chi]\int_{\Omega}u^{p-1}v^{-q-1}\nabla u\cdot\nabla v}
\\
&-\disp{(q(q+1)+pq\chi)\int_{\Omega}u^{p}v^{-q-2}|\nabla v|^2+(q+ap)\int_{\Omega}u^{p}v^{-q}}
\\
&-\disp{\mu p \int_{\Omega}u^{p+1}v^{-q}- q\int_{\Omega}u^{p+1}v^{-q-1}~~ \mbox{for all}~~  t\in(0,T_{max}),}
\\
\end{array}
\label{ssssscz2.511ssdddkkkkdfgg41ssfffffsddggggddfgddfggggddffff14}
\end{equation}
where by the  Young inequality,
\begin{equation}
\begin{array}{rl}
&\disp{[p(p-1)\chi+2pq]\int_{\Omega}u^{p-1}v^{-q-1}\nabla u\cdot\nabla v}
\\
\leq&\disp{p(p-1)\int_{\Omega}u^{p-2}v^{-q}|\nabla u|^2+\frac{p[(p-1)\chi+2q]^2}{4(p-1)}\int_{\Omega}u^{p}v^{-q-2}|\nabla v|^2}
\\
\end{array}
\label{1111cz2.511ssdddkkkkdfgg41sddggggddfgddfggggddffff14}
\end{equation}
for all $t\in(0,T_{max}).$
Inserting \dref{1111cz2.511ssdddkkkkdfgg41sddggggddfgddfggggddffff14} into \dref{ssssscz2.511ssdddkkkkdfgg41ssfffffsddggggddfgddfggggddffff14} implies that
\begin{equation}
\begin{array}{rl}
&\disp{\frac{d}{dt}\int_{\Omega}u^{p}v^{-q}}
\\
\leq&\disp{\{\frac{p[(p-1)\chi+2q]^2}{4(p-1)}-q(q+1)-pq\chi\}\int_{\Omega}u^{p}v^{-q-2}|\nabla v|^2}
\\
&+\disp{[q+ap]\int_{\Omega}u^{p}v^{-q}-\mu p \int_{\Omega}u^{p+1}v^{-q}- q\int_{\Omega}u^{p+1}v^{-q-1}~~ \mbox{for all}~~  t\in(0,T_{max}).}
\\
\end{array}
\label{cz2.511ssdddkkkkdfgg41sdgghdfffhdggggddfssdddgddfggggddffff14}
\end{equation}
 Denote
$$g(p;q,\chi):=\frac{p[(p-1)\chi+2q]^2}{4(p-1)}-pq\chi-q(q+1),$$
and rewrite it as the quadric expression in $q$ that
$$4(p-1)g(p;q,\chi):=\tilde{g}(q)=-4q^2+4(p-1)q-p(p-1)^2\chi^2.$$
According to $\Delta_{\tilde{g},q}:= 16(p - 1)^2(1 - p\chi^2) > 0$, our assumption $q \in (q_{2,-},q_{2,+})$ ensures that
$$g(p;q,\chi)< 0,$$ where
$q_{2,\pm}$ is given by \dref{czddfff2.511ssdddkkddfddfffffffgggdffffkkdfgsdffffg41sddggggddfgddfggggddffff14}.
Combine with \dref{cz2.511ssdddkkkkdfgg41sdgghdfffhdggggddfssdddgddfggggddffff14} to get
\begin{equation}
\begin{array}{rl}
&\disp{\frac{d}{dt}\int_{\Omega}u^{p}v^{-q}\leq (q+ap)\int_{\Omega}u^{p}v^{-q}-\mu p \int_{\Omega}u^{p+1}v^{-q}- q\int_{\Omega}u^{p+1}v^{-q-1}~~ \mbox{for all}~~  t\in(0,T_{max}).}
\\
\end{array}
\label{cz2.511ssdddkkkkdfgsdffffg41sddggggddfgddfggggddffff14}
\end{equation}
We now invoke the Young inequality  and use \dref{eqxhddfffgfgghhj45xx12112} in estimating
\begin{equation}
\begin{array}{rl}
(q+ap+1)\disp\int_{\Omega}u^{p}v^{-q} \leq &\disp{\mu p\int_{\Omega}u^{p+1}v^{-q}+\frac{1}{p+1}(\mu (q+1))^{-p}(q+ap+1)^{p+1} \int_{\Omega}v^{-q}}
\\
\leq &\disp{\mu p\int_{\Omega}u^{p+1}v^{-q}+\frac{1}{p+1}(\mu (q+1))^{-p}(q+ap+1)^{p+1} \eta_0^{-q}|\Omega|,}
\\
\end{array}
\label{cz2.511ssdddkkkddfffkdfgsdffffg41sddggggddfgddfggggddffff14}
\end{equation}
which together with \dref{cz2.511ssdddkkkkdfgsdffffg41sddggggddfgddfggggddffff14} implies
\begin{equation}
\begin{array}{rl}
&\disp{\frac{d}{dt}\int_{\Omega}u^{p}v^{-q}+\int_{\Omega}u^{p}v^{-q}\leq \frac{1}{p+1}(\mu (q+1))^{-p}(q+ap+1)^{p+1} \eta_0^{-q}|\Omega|~~ \mbox{for all}~~  t\in(0,T_{max}).}
\\
\end{array}
\label{cz2.511ssdddkkkssddddkdfgsdffffg41sddggggddfgddfggggddffff14}
\end{equation}
For all $t\in(0,T_{max}),$ integrating this between $0$ and $t$, taking into account Lemma \ref{lemma630} we obtain
\begin{equation}
\begin{array}{rl}
\disp\int_{\Omega}u^{p}(\cdot,t)v^{-q}(\cdot,t)\leq &\disp{e^{-t}\int_{\Omega}u^{p}_0v^{-q}_0 +
\frac{1}{p+1}(\mu (q+1))^{-p}(q+ap+1)^{p+1} \eta_0^{-q}|\Omega|(1-e^{-t}).}
\\
\end{array}
\label{cz2.511ssdddkkkssddddkdfgsdffffg41skkkkddggggddfgddfggggddffff14}
\end{equation}
Therefore,
 \dref{cz2.511ssdddkkddfddfffffffgggdffffkkdfgsdffffg41sddggggddfgddfggggddffff14} holds due to $p>1.$
%
%
\end{proof}

\begin{corollary}\label{corollaryssssssa45ddfffff630223}
Assume that $0<\chi<\sqrt{\frac{2}{N}}$ with $N\geq2$. Then there exist $\kappa>\frac{N}{2}$,
$q_0\in (0,\frac{N}{2})$
and $C_0>0$ independent of $\mu$
  such that
\begin{equation}
\begin{array}{rl}
&\disp{\int_{\Omega}u^{\kappa}v^{-q_0}\leq \frac{C_0}{\mu}~~~\mbox{for all}~~t\in(0,T_{max}).}
\\
\end{array}
\label{cz2.511ssdddkkddfddfssdddddffffffgggdffffkkdfgsdffffg41sddggggddfgddfggggddffff14}
\end{equation}
\end{corollary}
\begin{proof} Firstly, we derive that  $\frac{1}{\chi^2}>\frac{N}{2}$
by $N\geq2$ and $0<\chi<\sqrt{\frac{2}{N}}$. Therefore, we may choose $p:=\kappa>\frac{N}{2}$ such that $p\in(1,\frac{1}{\chi^2})$  and $q_0\in (\frac{p-1}{2}(1-\sqrt{1-p\chi^2}), \frac{p-1}{2}(1+\sqrt{1-p\chi^2}))\subset(0,\frac{N}{2})$.
The claimed inequality \dref{cz2.511ssdddkkddfddfssdddddffffffgggdffffkkdfgsdffffg41sddggggddfgddfggggddffff14} thus results from  Lemma \ref{lemkddddddkkkmsssa45630223}.
%
%
\end{proof}

\subsection{The proof of Theorem \ref{theorem3}}


The goal of this subsection is to establish a bound for $u$ with respect to the norm in $L^\infty(\Omega)$ in quantitative dependence on
a supposedly known pointwise lower bound for $v$. Indeed,
using that boundedness properties of $u$ (see  Corollary \ref{corollaryssssssa45ddfffff630223} ) and a pointwise lower bound for $v$ (see Lemma \ref{lemmssssdddssssssa45ddfffff630223}) imply boundedness properties of $\|u(\cdot,t)\|_{L^{\infty}(\Omega)}$ and $\|\nabla v(\cdot,t)\|_{L^{\infty}(\Omega)}$
by the variation-of-constants formula.

\begin{lemma}\label{lemkkkkmssskklllla45630223}Let $\Omega\subset\mathbb{R}^N(N\geq1)$ be a smooth bounded  domain.
 Let $a$ satisfy  \dref{x1sssss.731426kkklll677gg} and $(u,v)$ be a solution to \dref{111.ssderrfff1} on $(0,T_{max})$.
 If $\chi$
satisfies \dref{x1sssss.7314kkkkkk26kkklll677gg},
then
\begin{equation}
\begin{array}{rl}
&\disp{\sup_{t\in(0,T_{max})}(\|u(\cdot,t)\|_{L^{\infty}(\Omega)}+\|\nabla v(\cdot,t)\|_{L^{\infty}(\Omega)})<+\infty.}
\\
\end{array}
\label{cz2.511dfddfffgg41sddfgggddffff14}
\end{equation}
\end{lemma}
\begin{proof}
Firstly, according to Corollary \ref{corollaryssssssa45ddfffff630223} and Lemma \ref{wsdelemma45}, we pick $\kappa>\frac{N}{2}$ and $q_0\in (0,\frac{N}{2})$ such that 
\begin{equation}
\begin{array}{rl}
&\disp{\int_{\Omega}u^{\kappa}v^{-q_0}\leq C_1~~~\mbox{for all}~~t\in(0,T_{max}).}
\\
\end{array}
\label{1111cz2.511ssdddkkddfddfssdddddffffffgggdffffkkdfgsdffffg41sddggggddfgddfggggddffff14}
\end{equation}
holds with some $C_1>0.$ Since $q_0<\frac{N}{2}$ and $\kappa>\frac{N}{2}$, it is possible to fix $l_0\in(\frac{N}{2},\kappa)$ such that
$l_0<\frac{N(\kappa-q_0)}{N-2q_0}$. 
Using \dref{1111cz2.511ssdddkkddfddfssdddddffffffgggdffffkkdfgsdffffg41sddggggddfgddfggggddffff14}, we find that
\begin{equation}
\begin{array}{rl}
\left(\disp\int_{\Omega}u^{l_0}\right)^{\frac{1}{l_0}}\leq&\disp{\left(\int_{\Omega}u^{\kappa}v^{-q_0}\right)^{\frac{1}{\kappa}}
\left(\int_{\Omega}v^{\frac{l_0q_0}{\kappa-l_0}}\right)^{\frac{\kappa-l_0}{l_0\kappa}}}
\\
\leq&\disp{C_1^{\frac{1}{\kappa}}
\left(\int_{\Omega}v^{\frac{l_0q_0}{\kappa-l_0}}\right)^{\frac{\kappa-l_0}{\kappa}}}
\\
=&\disp{C_1^{\frac{1}{\kappa}}
\|v(\cdot,t)\|_{L^{{\frac{l_0q_0}{\kappa-l_0}}}(\Omega)}^\frac{q_0}{\kappa}~~~\mbox{for all}~~t\in(0,T_{max}),}
\\
\end{array}
\label{1111cz2.511ssdddkkddfddfddffssdddddffffffgggdffffkkdfgsdffffg41sddggggddfgddfggggddffff14}
\end{equation}
where  the
H\"{o}lder inequality has been used.
Since, $l_0<\frac{N(\kappa-q_0)}{N-2q_0}$ implies that
$$\frac{N}{2}[\frac{1}{l_0}-\frac{\kappa-l_0}{l_0q_0}]<1,$$
so that, by Lemma \ref{llssdrffmmggnnccvvccvvkkkkgghhkkllvvlemma45630}  with \dref{cz2.5ghju48cfg924ghyuji}, we have
\begin{equation}\begin{array}{rl}
&\| v(\cdot,t)\|_{L^\frac{l_0q_0}{\kappa-l_0}(\Omega)}\\
\leq&\disp \| e^{t(\Delta-1)}v_0\|_{L^\frac{l_0q_0}{\kappa-l_0}(\Omega)} +\int_{0}^{t}\| e^{(t-s)(\Delta-1)}(u(\cdot,s)-\bar{u})\|_{L^\frac{l_0q_0}{\kappa-l_0}(\Omega)} ds+\int_{0}^{t}\| e^{(t-s)(\Delta-1)}\bar{u}\|_{L^\frac{l_0q_0}{\kappa-l_0}(\Omega)} ds\\
\leq&\disp c_1\| v_0\|_{L^\infty(\Omega)} +c_1\int_{0}^{t}(1+(t-s)^{-\frac{N}{2}(\frac{1}{l_0}-\frac{\kappa-l_0}{l_0q_0})})e^{-(\lambda_1+1)(t-s) }
\|u(\cdot,s)-\bar{u}\|_{L^{l_0}(\Omega)} ds\\
&\disp+c_1\lambda\int_{0}^{t}e^{-(\lambda_1+1)(t-s) } ds\\
\leq&\disp c_1\| v_0\|_{L^\infty(\Omega)} +c_1\sup_{s\in(0,T_{max})}\|u(\cdot,s)\|_{L^{l_0}(\Omega)}\int_{0}^{t}(1+(t-s)^{-\frac{N}{2}(\frac{1}{l_0}-
\frac{\kappa-l_0}{l_0q_0})})e^{-(\lambda_1+1)(t-s) } ds\\
&\disp+c_1\lambda\int_{0}^{t}e^{-(\lambda_1+1)(t-s) } ds\\
\leq&\disp  C_2(1+\sup_{s\in(0,T_{max})}\|u(\cdot,s)\|_{L^{l_0}(\Omega)})~~~\mbox{for all}~~t\in(0,T_{max}),\\
\end{array}
\label{22222ssdddd23455cz2.ddff5ghjsdddddffgggu48cfg9sdddd24ghyuji}
\end{equation}
where  $\bar{u} = \bar{u}(t) :=
\frac{1}{|\Omega|}\int_{\Omega}u$
and
$$C_2= c_1\max\left\{\| v_0\|_{L^\infty(\Omega)}+\frac{\lambda}{\lambda_1+1},\int_{0}^{\infty}(1+(t-s)^{\frac{N}{2}(1-\frac{p-\alpha}{q\alpha})})e^{-(\lambda_1+1)(t-s) }  ds\right\}.$$
Here  $c_1$ is the same as Lemma \ref{llssdrffmmggnnccvvccvvkkkkgghhkkllvvlemma45630}.
Therefore,
\begin{equation}\begin{array}{rl}
\sup_{t\in(0,T_{max})}\| v(\cdot,t)\|_{L^\frac{l_0q_0}{\kappa-l_0}(\Omega)}\leq  C_2(1+\sup_{s\in(0,T_{max})}\|u(\cdot,s)\|_{L^{l_0}(\Omega)}).\\
\end{array}
\label{22222ssdddssdddd23455cz2.ddff5ghjsdddddffgggu48cfg9sdddd24ghyuji}
\end{equation}
Therefore, there is $C_3 > 0$ fulfilling
\begin{equation}\begin{array}{rl}
\sup_{t\in(0,T_{max})}\| u(\cdot,t)\|_{L^{l_0}(\Omega)}\leq  C_3(1+\left(\sup_{s\in(0,T_{max})}\|u(\cdot,s)\|_{L^{l_0}(\Omega)}\right)^{\frac{q_0}{\kappa}})\\
\end{array}
\label{22222ssdddssdddd23dfffff455cz2.ddff5ghjsdddddffgggu48cfg9sdddd24ghyuji}
\end{equation}
by using \dref{1111cz2.511ssdddkkddfddfddffssdddddffffffgggdffffkkdfgsdffffg41sddggggddfgddfggggddffff14}.
Upon the observation that $\frac{q_0}{\kappa}<1$ due to $\kappa>\frac{N}{2}>q_0$, we can conclude
\begin{equation}
\begin{array}{rl}
&\disp{\sup_{t\in(0,T_{max})}\|u(\cdot,t)\|_{L^{{l_0}}(\Omega)}\leq\tilde{\lambda}.}
\\
\end{array}
\label{cz2.511ssddddsdffffgg41sddggggddddddddfgddfggggddffff14}
\end{equation}
%
%
Now,   collecting  \dref{cz2.5ghju48cfg924ghyuji} and \dref{cz2.511ssddddsdffffgg41sddggggddddddddfgddfggggddffff14},
we derive that for
some $r_0\geq 1$ 
 satisfying $r_0> \frac{N}{2}$,
\begin{equation}
\begin{array}{rl}
&\disp{\sup_{t\in(0,T_{max})}\|u(\cdot,t)\|_{L^{{r_0}}(\Omega)}\leq\tilde{\lambda}_1.}
\\
\end{array}
\label{cz2.511ssddddfgg41sddggggddddddddfgddfggggddffff1422}
\end{equation}
Now,
 involving the variation-of-constants formula
for $v$ and $L^p$-$L^q$ estimates for the heat semigroup again, we derive that for $\theta\in[1,\frac{N r_0}{N- r_0}),$ there exists a positive constant $C_4$ such that
\begin{equation}\begin{array}{rl}
&\|\nabla v(\cdot,t)\|_{L^\theta(\Omega)}\\
\leq&\disp \|\nabla e^{t(\Delta-1)}v_0\|_{L^\theta(\Omega)} +\int_{0}^{t}\|\nabla e^{(t-s)(\Delta-1)}u(\cdot,s)\|_{L^\theta(\Omega)} ds\\
\leq&\disp c_2\|\nabla v_0\|_{L^\infty(\Omega)} +c_2\int_{0}^{t}(1+(t-s)^{-\frac{1}{2}-\frac{N}{2}(\frac{1}{ r_0}-\frac{1}{\theta})})e^{-\lambda_1(t-s) }\|u(\cdot,s)\|_{L^ {r_0}(\Omega)} ds\\
\leq&\disp c_2\|\nabla v_0\|_{L^\infty(\Omega)} +c_2\tilde{\lambda}_1 \int_{0}^{t}(1+(t-s)^{-\frac{1}{2}-\frac{N}{2}(\frac{1}{ {r_0}}-\frac{1}{\theta})})e^{-\lambda_1(t-s) } ds\\
\leq&\disp  C_4~~~\mbox{for all}~~t\in(0,T_{max})\\
\end{array}
\label{1112ssdddd23455cz2.ddff5ghjsdddddffgggu48cfg924ghyuji}
\end{equation}
by combining \dref{cz2.511ssddddfgg41sddggggddddddddfgddfggggddffff1422} with Lemma \ref{llssdrffmmggnnccvvccvvkkkkgghhkkllvvlemma45630},
where  $c_2$ is the same as Lemma \ref{llssdrffmmggnnccvvccvvkkkkgghhkkllvvlemma45630}.
Here we have used the fact that
$$\int_{0}^{t}(1+(t-s)^{-\frac{1}{2}-\frac{N}{2}(\frac{1}{ {r_0}}-\frac{1}{\theta})})e^{-\lambda_1(t-s) } ds\leq \int_{0}^{\infty}(1+s^{-\frac{1}{2}-\frac{N}{2}(\frac{1}{ {r_0}}-\frac{1}{\theta})})e^{-\lambda_1s } ds<+\infty.$$
Therefore, there  is  $C_5 > 0$ satisfies 
\begin{equation}
\int_{\Omega}|\nabla v|^{\theta} \leq C_5~~\mbox{for all}~~ t\in(0, T_{max})~~~\mbox{and}~~~\theta\in[1,\frac{N {r_0}}{N- {r_0}})
\label{2ssdddd23455cz2.ddff5ghjsdddddffgggu48cfg924ghyuji}
\end{equation}
by \dref{1112ssdddd23455cz2.ddff5ghjsdddddffgggu48cfg924ghyuji}.
Next, fix $T\in (0, T_{max})$, let $M(T):=\sup_{t\in(0,T)}\|u(\cdot,t)\|_{L^\infty(\Omega)}$ and $\tilde{h} :=\nabla v $. Then by  \dref{2ssdddd23455cz2.ddff5ghjsdddddffgggu48cfg924ghyuji},
there exists $C_6 > 0$ such that
\begin{equation}
\begin{array}{rl}
\|\tilde{h} (\cdot, t)\|_{L^{\theta}(\Omega)}\leq&\disp{C_6~~ \mbox{for all}~~ t\in(0,T_{max})~~\mbox{and some }~~N<\theta_0<\frac{N {r_0}}{N- {r_0}}.}\\
\end{array}
\label{cz2ddff.57151ccvhhjjjkkkuuifghhhivhccvvhjjjkkhhggjjllll}
\end{equation}
Next, 
by means of an
associate variation-of-constants formula once again, one  can derive that  for any $t\in(t_0, T)$,
\begin{equation}
u (t)=e^{(t-t_0)\Delta}u (\cdot,t_0)-\chi\int_{t_0}^{t}e^{(t-s)\Delta}\nabla\cdot(\frac{u (\cdot,s)}{v (\cdot,s)}\tilde{h} (\cdot,s)) ds+\int_{t_0}^{t}e^{(t-s)\Delta}(au(\cdot,s)-\mu u^2(\cdot,s))ds,
\label{5555fghbnmcz2.5ghjjjkkklu48cfg924ghyuji}
\end{equation}
where $t_0 := (t-1)_{+}$.
If $t\in(0,1]$,
by virtue of the maximum principle, we derive that
\begin{equation}
\begin{array}{rl}
\|e^{(t-t_0)\Delta}u (\cdot,t_0)\|_{L^{\infty}(\Omega)}\leq &\disp{\|u_0\|_{L^{\infty}(\Omega)},}\\
\end{array}
\label{zjccffgbhjffghhjcghhhjjjvvvbscz2.5297x96301ku}
\end{equation}
while if $t > 1$, we estimate the  first integral on the right  of \dref{5555fghbnmcz2.5ghjjjkkklu48cfg924ghyuji} by means of the Neumann heat semigroup and Lemma \ref{llssdrffmmggnnccvvccvvkkkkgghhkkllvvlemma45630} according to
 %
%
\begin{equation}
\begin{array}{rl}
\|e^{(t-t_0)\Delta}u (\cdot,t_0)\|_{L^{\infty}(\Omega)}\leq &\disp{C_7(t-t_0)^{-\frac{N}{2}}\|u (\cdot,t_0)\|_{L^{1}(\Omega)}\leq C_8.}\\
\end{array}
\label{zjccffgbhjffghhjcghghjkjjhhjjjvvvbscz2.5297x96301ku}
\end{equation}
Now,  in view of \dref{cz2.5ghju48cfg924ghyuji} and \dref{eqxhddfffgfgghhj45xx12112},  we fix an arbitrary $p\in(N,\theta)$ and then once more invoke known smoothing
properties of the
Stokes semigroup  and the H\"{o}lder inequality to find $C_9 > 0$ such that
\begin{equation}
\begin{array}{rl}
&\disp \chi\int_{t_0}^t\| e^{(t-s)\Delta}\nabla\cdot(\frac{u (\cdot,s)}{v (\cdot,s)}\tilde{h} (\cdot,s)\|_{L^\infty(\Omega)}ds\\
\leq&\disp C_9\int_{t_0}^t(1+(t-s)^{-\frac{1}{2}-\frac{N}{2p}})e^{-\lambda_1(t-s) }\|u (\cdot,s)\tilde{h} (\cdot,s)\|_{L^p(\Omega)}ds\\
\leq&\disp C_9\int_{t_0}^t(1+(t-s)^{-\frac{1}{2}-\frac{N}{2p}})e^{-\lambda_1(t-s) }\| u (\cdot,s)\|_{L^{\frac{p\theta}{\theta-p}}(\Omega)}\|\tilde{h} (\cdot,s)\|_{L^{\theta}(\Omega)}ds\\
\leq&\disp C_9\int_{t_0}^t(1+(t-s)^{-\frac{1}{2}-\frac{N}{2p}})e^{-\lambda_1(t-s) }\| u (\cdot,s)\|_{L^{\infty}(\Omega)}^b\| u (\cdot,s)\||_{L^1(\Omega)}^{1-b}\|\tilde{h} (\cdot,s)\|_{L^{\theta}(\Omega)}ds\\
\leq&\disp C_{10}M^b(T)~~\mbox{for all}~~ t\in(0, T),\\
\end{array}
\label{ccvbccvvbbnnndffghhjjvcvvbccfbbnfgbghjjccmmllffvvggcvvvvbbjjkkdffzjscz2.5297x9630xxy}
\end{equation}
where $b:=\frac{p\theta-\theta+p}{p\theta}\in(0,1)$
and
$$C_{10}:=C_9\lambda^{1-b}C_6\int_{0}^{1}(1+\sigma^{-\frac{1}{2}-\frac{N}{2p}})e^{-\lambda_1\sigma }d\sigma.$$
Since $p>N$, we conclude that
$-\frac{1}{2}-\frac{N}{2p}>-1$.
Similarly, due to Lemma \ref{llssdrffmmggnnccvvccvvkkkkgghhkkllvvlemma45630},  we can estimate the third  integral on the right of \dref{5555fghbnmcz2.5ghjjjkkklu48cfg924ghyuji}  as follows:
\begin{equation}
\begin{array}{rl}
\disp\int_{t_0}^t \|e^{(t-s)\Delta}(au(\cdot,s)-\mu u^2(\cdot,s))\|_{L^\infty(\Omega)}ds\leq&
\disp\int_{t_0}^t\sup_{u\geq0}(au-\mu u^2)_{+}ds\\
\leq&\disp\int_{t_0}^ta_{+}^{2}\mu^{-1}2^{-2}\\
\leq&\disp \frac{a^{2}}{4\mu}.\\
\end{array}
\label{ccvbccvvbbnnndffghhjdfghhhhjvcvvbccfbbnfgbghjjccmmllffvvvvvbbjjfzjscz2.5297x9630xxy}
\end{equation}
so that, in view of the definition of $M(T)$,
there exists a positive $C_{11}$ such that
%
\begin{equation}
\begin{array}{rl}
&\disp  M(T)\leq C_{11}+C_{11}M^b(T)~~\mbox{for all}~~ T\in(0, T_{max})\\
\end{array}
\label{ccvbccvvbbnnndffghhjjvcvvfghhhbccfbbnfgbghjjccmmllffvvggcvvvvbbjjkkdffzjscz2.5297x9630xxy}
\end{equation}
by using \dref{5555fghbnmcz2.5ghjjjkkklu48cfg924ghyuji}--\dref{ccvbccvvbbnnndffghhjjvcvvbccfbbnfgbghjjccmmllffvvggcvvvvbbjjkkdffzjscz2.5297x9630xxy}.
By comparison, this implies that
\begin{equation}
\begin{array}{rl}
\|u (\cdot, t)\|_{L^{\infty}(\Omega)}\leq&\disp{C_{12}~~ \mbox{for all}~~ t\in(0,T_{max}),}\\
\end{array}
\label{cz2.57ghhhh151ccvhhjjjkkkffgghhuuiivhccvvhjjjkkhhggjjllll}
\end{equation}
due to $b<1$ and $T\in (0, T_{max})$ was arbitrary.
Finally, with the regularity properties from \dref{cz2.57ghhhh151ccvhhjjjkkkffgghhuuiivhccvvhjjjkkhhggjjllll} at hand, one can readily
derive
\begin{equation}
\begin{array}{rl}
\| v (\cdot, t)\|_{W^{1,\infty}(\Omega)}\leq C_{13}~~ \mbox{for all}~~ t\in(0,T_{max})\\
\end{array}
\label{zjccffgbhjcvvvbscz2.5297x96301ku}
\end{equation}
 by means of standard parabolic regularity arguments applied to the second equation in
\dref{111.ssderrfff1}.
The proof Lemma \ref{lemkkkkmssskklllla45630223} is completed.
\end{proof}


We are now in a position to prove Theorem \ref{theorem3}.


{\bf The proof of Theorem \ref{theorem3}}~
In view of \dref{cz2.511dfddfffgg41sddfgggddffff14}, we apply Lemma \ref{lemma70} to reach a contradiction.
 Hence the  classical solution $(u,v)$ of \dref{111.ssderrfff1} is global in time and bounded. Finally, employing  the same arguments as in the proof of Lemma 1.1 in \cite{Winkler37103}, and taking advantage of Lemma \ref{lemkkkkmssskklllla45630223}, we conclude the uniqueness of solution to \dref{111.ssderrfff1}.

 \section{Asymptotic behavior}
In this section we study the long-time behavior for \dref{111.ssderrfff1} in the case $\mu$ is large enough.
%
The goal of this section will be to establish the convergence properties stated in
Theorem \ref{theorem3}. The key idea of our approach is to use the variation-of-constants formula, the form of which is inspired by \cite{Winkler79312} (see also \cite{Tao79477ddffvg,Zhengssddghhhhjmaa22}).
To show the global asymptotic stability of $(\frac{a}{\mu }, \frac{a}{\mu }),$ it will be convenient to introduce the following notation:
 \begin{equation}
U(x,t)=\frac{\mu }{a}u(x,t)~~\mbox{and}~~V(x,t)=v(x,t)-\frac{a}{\mu }.
\label{ddddffffffhddffgggghhhzaxscdfv1.1}
\end{equation}
Accordingly, 
  we see $(U, V)$ have the following properties:
 \begin{equation}
 \left\{\begin{array}{ll}
    U_t=\Delta U-\chi \nabla\cdot(\frac{U}{v}\nabla V)+aU(1-U),~~x\in \Omega, t>0,\\
 \disp{ V_t=\Delta V- V+\frac{a}{\mu }(U-1)},~~x\in \Omega, t>0,\quad\\
 \disp{\frac{\partial U}{\partial \nu}=\frac{\partial V}{\partial \nu}=0},\quad
x\in \partial\Omega, t>0,\\
\disp{U(x,0):=U_0(x)=\frac{\mu }{a}u_0(x),V(x,0):=V_0(x)=v(x,t)-\frac{a}{\mu },~
x\in \Omega}\\
 \end{array}\right.\label{ddddffffffhhhhzaxscdfv1.1}
\end{equation}
by \dref{111.ssderrfff1} and a
  straightforward computation.

From the proof of Lemma  \ref{lemmssssdddssssssa45ddfffff630223}, we derive that: 
there exists a positive constant $k_0$  independent of $\mu$ such that
 \begin{equation}\label{eqddddxhddfffgfggdfghhhhj45xx12112}
\frac{1}{v^2}\leq  k_0~~\mbox{for all}~~x\in\Omega~~~\mbox{and}~~~t>0,
\end{equation}
where $k_0=\frac{1}{\eta_0^2}$.

The crucial idea of the proof of Theorem \ref{theoremssdddffggdd3} is to show   a  Lyapunov functional for  \dref{ddddffffffhhhhzaxscdfv1.1} under  large enough of  $\mu$.
The main idea of Theorem \ref{theoremssdddffggdd3} comes from the
proof of Lemma 3.7 of \cite{Tao79477ddffvg}.

\begin{lemma}\label{lemmaddffffdfffgg4dddd5630}
 Let $(u, v)$ be a global classical solution of \dref{111.ssderrfff1}. 
 Then if 
\begin{equation}\mu>\max\{1,a\chi k_0\frac{\sqrt{2}}{4}\},\label{dffffcz2.51141ssdderfttddfggt14}
\end{equation}
then for all $t > 0$ the function
\begin{equation}
\begin{array}{rl}
&\disp{F(t):=\int_{\Omega}(U-1-\ln U)}+\frac{L}{2} \disp\int_{\Omega}V^2
\\
\end{array}
\label{cz2.51141ssdderfttddfggt14}
\end{equation}
satisfies
\begin{equation}
\begin{array}{rl}
&\disp{F'(t)\leq -G(t)}
\\
\end{array}
\label{cz2.51141ssdddgggderfttddfggt14}
\end{equation}
with
\begin{equation}
G(t)=G_0(\int_{\Omega}(U-1)^2+\frac{L}{2}\int_{\Omega}V^2)\label{cz2.511ffff41ssdddgggderdffffttddfggt14}
\end{equation}
and
$$G_0=\min\{a-\frac{L}{2}(\frac{a}{\mu })^{{2}},L-\frac{\chi^2k_0}{4}\}>0,$$
where $L$ is a  positive constant which satisfies that
\begin{equation}\frac{\chi^2k_0}{4}<L<\frac{2\mu^2}{a^2}
\label{dffzjscz2.5297xssedrttggg9ghhhh6ddfgggdddd30ffgggxxy}
\end{equation}
and
$k_0$ is the same as \dref{eqddddxhddfffgfggdfghhhhj45xx12112}.
\end{lemma}
\begin{proof}
Firstly,  multiplying the second equation in \dref{ddddffffffhhhhzaxscdfv1.1} by $V$, we obtain
\begin{equation}
\begin{array}{rl}
\disp\frac{1}{2}\disp\frac{d}{dt}\int_{\Omega}V^2+\disp\int_{\Omega}|\nabla V|^2+
\int_{\Omega}V^2=&\disp{
\frac{a}{\mu }\int_{\Omega}V(U-1) }\\
\leq&\disp{\frac{1}{2} \int_{\Omega}V^2+\frac{1}{2}\left(\frac{a}{\mu }\right)^2\int_{\Omega}(U-1)^2~~ \mbox{for all}~~ t>0}\\
\end{array}
\label{cz2.51141sdfghhsdfffgddehhjjjjssddrfddffgggttt14}
\end{equation}
by using  the Young inequality.
Therefore, we derive from \dref{cz2.51141sdfghhsdfffgddehhjjjjssddrfddffgggttt14} that
\begin{equation}
\begin{array}{rl}
\disp\frac{1}{2}\disp\frac{d}{dt}\int_{\Omega}V^2+\disp\int_{\Omega}|\nabla V|^2+\frac{1}{2}\
\int_{\Omega}V^2\leq&\disp{ \frac{1}{2}\left(\frac{a}{\mu }\right)^2\int_{\Omega}(U-1)^2~~ \mbox{for all}~~ t>0.}\\
\end{array}
\label{cz2.51141sdfghhsddehhjjjjssddrfddffgggttt14}
\end{equation}
On the other hand,
the strong maximum principle along with the assumption $U_0\not\equiv 0$ yields
$U>0$ in $\bar{\Omega}\times(0, +\infty)$.
Relying on this,
we  multiply the first equation in \dref{ddddffffffhhhhzaxscdfv1.1}  by $1-\frac{1}{U}$ and
integrate by parts, then by 
 \dref{eqddddxhddfffgfggdfghhhhj45xx12112},
%
%
\begin{equation}
\begin{array}{rl}
&\disp{\frac{d}{dt}\int_{\Omega}(U-1-\ln U)}
\\
=&\disp{-\int_{\Omega}\frac{|\nabla U|^2}{U^2}+\chi\int_{\Omega}\frac{1}{Uv}\nabla U\cdot\nabla v
-a\int_{\Omega}(U-1)^2}\\
\leq&\disp{-\int_{\Omega}\frac{|\nabla U|^2}{U^2}+\int_{\Omega}\frac{|\nabla U|^2}{U^2}+\frac{\chi^2}{4}\int_{\Omega}\frac{|\nabla V|^2}{v^2}-
a\int_{\Omega}(U-1)^2}\\
\leq&\disp{\frac{\chi^2k_0}{4}\int_{\Omega}|\nabla V|^2
-a\int_{\Omega}(U-1)^2~~ \mbox{for all}~~ t>0}\\
\end{array}
\label{cz2.51141ssdderfttt14}
\end{equation}
by the Young inequality, where $k_0$ is the same as \dref{eqddddxhddfffgfggdfghhhhj45xx12112}.
Observe that \dref{dffzjscz2.5297xssedrttggg9ghhhh6ddfgggdddd30ffgggxxy},
let
$\dref{cz2.51141sdfghhsddehhjjjjssddrfddffgggttt14}\times L
+\dref{cz2.51141ssdderfttt14}$, then we deduce
\begin{equation}
\begin{array}{rl}
&\disp{\frac{d}{dt}\int_{\Omega}(U-1-\ln U)+(a-\frac{L}{2}(\frac{a}{\mu })^{{2}})\int_{\Omega}(U-1)^2}\\
&\disp{+\frac{L}{2} \disp\frac{d}{dt}\int_{\Omega}V^2+(L-\frac{\chi^2k_0}{4})\int_{\Omega}|\nabla V|^2+\frac{L}{2}
\int_{\Omega}V^2}
\\
\leq&\disp{0~~ \mbox{for all}~~ t>0,}\\
\end{array}
\label{cz2.51141ssdderfttddfggt14}
\end{equation}
which together with the definition  of $F$ and $G$   implies that \dref{cz2.51141ssdddgggderfttddfggt14} holds.
\end{proof}

\begin{lemma}\label{fghfbglemma4563025xxhjklojjkkkgyhuissddff}
Assume that  the conditions in Theorem \ref{theorem3} are satisfied.
 Let $(u, v)$ be a global classical solution of \dref{111.ssderrfff1}. 
There is $\alpha > 0$ such that $u, v \in C^{\alpha,\frac{\alpha}{2}}(\bar{\Omega}\times( 1,+\infty))$. Moreover, there exists a positive constant
$C$ such that for
every $t>  1$,
\begin{equation}
\begin{array}{rl}\label{gbhngdddfffdfgggggeqx45xx1211}
&\disp{\|u(\cdot,t)\|_{C^{\alpha,\frac{\alpha}{2}}(\bar{\Omega}\times( 1,+\infty))}+\|v(\cdot,t)\|_{C^{\alpha,\frac{\alpha}{2}}(\bar{\Omega}\times( 1,+\infty))}\leq C}
\\
\end{array}
\end{equation}
and
\begin{equation}
\begin{array}{rl}\label{gbhngddfgggggeqx45xx1211}
&\disp{\|u(\cdot,t)\|_{W^{1,\infty}(\Omega)}+\|v(\cdot,t)\|_{W^{1,\infty}(\Omega)}\leq C.}
\\
\end{array}
\end{equation}
\end{lemma}
\begin{proof}
Firstly,  based on the regularity of $u$ and $v$, one can readily get a constant $C_1> 0$  such that
\begin{equation}
\| u(\cdot,t)\|_{L^{\infty}(\Omega)}+\| v(\cdot,t)\|_{W^{1,\infty}(\Omega)}\leq C_1~~\mbox{for all}~~ t>0.
\label{ddxcvbbggdddfghhdfgcz2vv.5ghju48cfg924ghddfgggyusdffji}
\end{equation}
Next,  we can rewrite the first equation of \dref{111.ssderrfff1} as
\begin{equation}u_t=\nabla a(x,t,u,\nabla u)+b(x,t,u)
\label{ddxcvbbggdddfghhdfgcz2vv.5ghju48cfg924ghddfgggyddfggusdffji}
\end{equation}
with boundary data $a(x,t,u,\nabla u)\cdot\nu = 0$ on $\partial\Omega\times(0,\infty)$
, where $a(x,t,u,\nabla u) := \nabla u-\frac{u}{ v}\nabla v$, $b(x,t,u) = au-\mu u^2$, $(x, t)\in
 \Omega\times(0,\infty).$
 Therefore, in view of \dref{eqddddxhddfffgfggdfghhhhj45xx12112} and \dref{ddxcvbbggdddfghhdfgcz2vv.5ghju48cfg924ghddfgggyusdffji}, applying Lemma 1.3 of \cite{Porzio710} to \dref{ddxcvbbggdddfghhdfgcz2vv.5ghju48cfg924ghddfgggyddfggusdffji}, we drive that
 \begin{equation}u \in C^{\alpha,\frac{\alpha}{2}}(\bar{\Omega}\times[1,+\infty)),
 \label{ddxcvbbggdddfghhdfgcz2vv.5ghju48cfg9ddffff24ghddfgggyddfggusdffji}
\end{equation}
so that, by  the second equation of \dref{111.ssderrfff1}, we can get that $v\in C^{\alpha,1+\frac{\alpha}{2}}(\bar{\Omega}\times[1,+\infty))$.
Finally, with the aforementioned
regularity properties of $u$ and $v$ at hand, we can obtain form Theorem IV.5.3 of \cite{Ladyzenskaja710}
that
\dref{gbhngddfgggggeqx45xx1211}  holds.
\end{proof}

\begin{lemma}\label{fhhghfbgddddlemma4563025xxhjklojjkkkgyhuissddff}
Assume the hypothesis of Theorem \ref{theoremssdddffggdd3}
holds.
 Then if $(u, v)$ is a nonnegative
global classical solution of \dref{111.ssderrfff1}, we have
\begin{equation}\label{fvgbccvvhnjmkfgbeerrrdffrhnn6291}
\lim_{t\rightarrow+\infty}\|U(\cdot,t)-1\|_{L^\infty(\Omega)}=0
\end{equation}
as well as
\begin{equation}\label{fvgbccvvhnjmkfgbeddffferrrdffrhnddfn6291}
\lim_{t\rightarrow+\infty}\|V(\cdot,t)\|_{L^\infty(\Omega)}=0.
\end{equation}
\end{lemma}
\begin{proof}
Starting from the functional inequality \dref{cz2.51141ssdddgggderfttddfggt14} and Lemma \ref{fghfbglemma4563025xxhjklojjkkkgyhuissddff},
Lemma \ref{fhhghfbgddddlemma4563025xxhjklojjkkkgyhuissddff}
can be proved in the same way as in Ref. \cite{Tao79477ddffvg}. Therefore, we omit it here.
\end{proof}
With the above preparation, we can now integrate the energy inequality (see Lemma \ref{lemmaddffffdfffgg4dddd5630}) and make use of the Gagliardo-Nirenberg inequality as well as Lemma \ref{fhhghfbgddddlemma4563025xxhjklojjkkkgyhuissddff} to achieve that the solution $(u, v)$ exponentially
converges to the constant stationary solution $(\frac{a}{\mu },
\frac{a}{\mu })$ in the norm of $L^\infty(\Omega)$ as
$t\rightarrow\infty$.


{\bf The proof of Theorem \ref{theoremssdddffggdd3}}~
\begin{proof}
Denote $h(s):= s-1-\ln s.$ Noticing that
$h'(s) = 1-\frac{1}{s}$ and $h''(s) = 1+s^{-2} > 0$ for all $s > 0$, we obtain that $h(s) \geq h(1) = 0$ and $F(t)$ is nonnegative.
From Lemma \ref{lemmaddffffdfffgg4dddd5630}, we have 
\begin{equation}
\begin{array}{rl}
\disp\disp\int_{\tau_0+1}^tG(s)\leq&\disp{ F(\tau_0+1)-F(t)}\\
\leq&\disp{ F(\tau_0+1)~~\mbox{for all}~~t>\tau_0+1,}\\
\end{array}
\label{cz2.51141sdfghhsddegghhhhhjjjjssddrfddffgggttt14}
\end{equation}
using the definition of $G$ and $F$ again, we have
\begin{equation}
\int_{\tau_0+1}^{t}\left\{\int_{\Omega}(U-1)^2+\frac{L}{2}\int_{\Omega}V^2\right\}<+\infty.
\label{cz2.511ffff41ssddddddddgggderdffffttddfggt14}
\end{equation}
Observe that
$$\lim_{s\rightarrow1}\frac{s-1-\ln s}{(s-1)^2}=\frac{1}{2},$$
so that, for $\varepsilon=\frac{1}{6}$, there exists a positive constant $\delta (\delta<\frac{1}{4})$ such that
for any $|s-1|<\delta$,
$$-\frac{1}{6}<\frac{s-1-\ln s}{(s-1)^2}-\frac{1}{2}<\frac{1}{6},$$
thus,
$$\frac{1}{3}{(s-1)^2}<s-1-\ln s<\frac{2}{3}{(s-1)^2}~~~~\mbox{for any}~~|s-1|<\delta.$$
For the above $\delta>0,$ 
then there exists $t_0 > 0$ such that for all $t > t_0$,
\begin{equation}\label{fvgbdffffggccvvhnjmkfgssdddbedddfddffferrrdffrfffhnn6291}
\|U(\cdot,t)-1\|_{L^\infty(\Omega)} < \delta\end{equation}
by \dref{fvgbccvvhnjmkfgbeerrrdffrhnn6291}.
Therefore, \dref{fvgbdffffggccvvhnjmkfgssdddbedddfddffferrrdffrfffhnn6291} implies that
for all $x\in\Omega$ and $t> t_0,$
\begin{equation}\label{fvgbdffffggccvvhnjmkfgbedddfddffferrrdffrfffhnn6291}
\frac{1}{3}{(U(x,t)-1)^2}<U(x,t)-1-\ln U(x,t)<\frac{2}{3}{(U(x,t)-1)^2}\leq{(U(x,t)-1)^2},
\end{equation}
which in view of  the definition of $F$ and $G$
yields to
\begin{equation}
\begin{array}{rl}
\disp\frac{1}{3}\int_{\Omega}(U-1)^2+\frac{L}{2}\int_{\Omega}V^2\leq
F(t)\leq \frac{1}{G_0}G(t).
\end{array}
\label{cz2.511ffffddffff41ssdddgggderdffffttddfggt14}
\end{equation}
Hence
\begin{equation}
\begin{array}{rl}
\disp F'(t)\leq-
G(t)\leq G_0F(t),
\end{array}
\label{cz2.511ffffddffff41ssdddgggdeddffgrdffffttddfggt14}
\end{equation}
from which one has
\begin{equation}
\begin{array}{rl}
\disp F(t)\leq
F(t_0)e^{-G_0(t-t_0)},
\end{array}
\label{cz2.511ffffddhhhffff41ssdddgggdeddffgrdffffttddfggt14}
\end{equation}
Substituting \dref{cz2.511ffffddhhhffff41ssdddgggdeddffgrdffffttddfggt14} into \dref{cz2.511ffffddffff41ssdddgggderdffffttddfggt14}), we obtain
\begin{equation}
\begin{array}{rl}
\disp\frac{1}{3}\int_{\Omega}(U-1)^2+\frac{L}{2}\int_{\Omega}V^2\leq
F(t_0)e^{-G_0(t-t_0)},
\end{array}
\label{cz2.511ffffddffff41ssdddggddfffffgderdffffttddfggt14}
\end{equation}
which implies that there is  $C_1 > 0$ fulfilling
such that 
\begin{equation}\label{fvgbccvvhnjmkfgbdffrhnkkkn6291}
\|U(\cdot,t)-1\|_{L^2(\Omega)}\leq C_1e^{-\frac{G_0}{2} t}~~~\mbox{for all}~~t > t_0
\end{equation}
as well as
\begin{equation}\label{ffgfvgbccvvhnjmkfghhjjgllllbhnn6291}
\|v(\cdot,t)-\frac{a}{\mu}\|_{L^2(\Omega)}\leq C_1e^{-\frac{G_0}{2} t}~~~\mbox{for all}~~t > t_0.
\end{equation}
Furthermore,  we also derive that  there exist constants $C_2 > 0$ and $t_0>1$ such that
\begin{equation}
\begin{array}{rl}\label{gbhngddfgggggeqx45xx1ssddd211}
&\disp{\|U(\cdot,t)-1\|_{W^{1,\infty}(\Omega)}+\|V(\cdot,t)\|_{W^{1,\infty}(\Omega)}\leq C_2~~\mbox{for all}~~t>t_0}
\\
\end{array}
\end{equation}
by using \dref{gbhngddfgggggeqx45xx1211} and \dref{ddddffffffhddffgggghhhzaxscdfv1.1}.
We also recall from the Gagliardo-Nirenberg inequality that there exist positive constants $C_4$ and $C_5$ such that
\begin{equation}\label{fvgbccvvhnjmkfhhhhhgbdffrhnkkkn6291}
\begin{array}{rl}
\disp
\|U(\cdot,t)-1\|_{L^\infty(\Omega)}\leq&{ C_3(\|U(\cdot,t)-1\|_{W^{1,\infty}(\Omega)}^{\frac{N}{N+2}}\|U(\cdot,t)-1\|_{L^2(\Omega)}^{\frac{2}{N+2}}+\|U(\cdot,t)-1\|_{L^2(\Omega)})}\\
\leq&{ C_4\|U(\cdot,t)-1\|_{L^2(\Omega)}^{\frac{2}{N+2}}}\\
\leq&{C_5e^{-\frac{G_0}{N+2} t}~~\mbox{for all}~~t>t_0.}\\
\end{array}
\end{equation}
Similarly, we can obtain
\begin{equation}\label{fvgbccvvhnjffggmkfhhhhhgbdffrhnkkkn6291}
\begin{array}{rl}
\disp
\|V(\cdot,t)\|_{L^\infty(\Omega)}\leq&{ C_6e^{-\frac{G_0}{N+2} t}~~\mbox{for all}~~t>t_0.}\\
\end{array}
\end{equation}
\end{proof}



{\bf Acknowledgement}:
This work is partially supported by  Shandong Provincial
Science Foundation for Outstanding Youth (No. ZR2018JL005), the National Natural
Science Foundation of China (No. 11601215)  and Project funded by China
Postdoctoral Science Foundation (No. 2019M650927, 2019T120168).
%

\end{document}